\documentclass{article}

\usepackage{geometry}
\usepackage[utf8]{inputenc}
\usepackage[english]{babel}
\usepackage{lmodern}
\usepackage{microtype}

\usepackage{graphicx}
\usepackage{hyperref}
\usepackage{xcolor}
\usepackage{standalone}
\usepackage{tikz}
\usepackage[capitalise]{cleveref}
\usepackage{authblk}

\usepackage{amsmath,amsfonts,amssymb}
\usepackage{amsthm}
\usepackage{bbm}

\newcommand{\N}{\ensuremath{\mathbb{N}}}
\newcommand{\Z}{\ensuremath{\mathbb{Z}}}
\newcommand{\Q}{\ensuremath{\mathbb{Q}}}

\newcommand{\F}{\ensuremath{{\mathbb{F}_2}}}
\newcommand{\K}{\ensuremath{R}}
\renewcommand\vec{\mathbf}
\newcommand{\FF}{\ensuremath{{\mathbb{F}_p}}}
\newcommand{\ones}{\mathbbm{1}\hspace*{-0.5mm}}

\newcommand{\Ann}{\ensuremath{\text{Ann}}}
\newcommand{\Res}{\ensuremath{{\text{Res}_X}}}
\newcommand{\ResY}{\ensuremath{{\text{Res}_Y}}}

\newtheorem{theorem}{Theorem}
\newtheorem*{theorem*}{Theorem}
\newtheorem{lemma}[theorem]{Lemma}
\newtheorem*{lemma*}{Lemma}
\newtheorem{corollary}[theorem]{Corollary}
\newtheorem{definition}{Definition}
\newtheorem*{conjecture}{Conjecture}

\newtheorem{prop}{Property}
\crefname{prop}{Property}{Properties}

\title{Nivat's Conjecture and pattern complexity in algebraic subshifts\footnote{Research supported by the Academy of Finland grant 296018.}}
\date{}

\author[1]{Jarkko Kari}
\author[1,2]{Etienne Moutot}
\affil[1]{Department of Mathematics and Statistics, University of Turku}
\affil[2]{LIP, ENS de Lyon -- CNRS -- UCBL -- Universit\'e de Lyon}

\begin{document}
\maketitle

\begin{abstract}
We study Nivat's conjecture on algebraic subshifts and prove that
in some of them every low complexity configuration is periodic. This is the case
in the Ledrappier subshift (the 3-dot system) and, more generally, in all two-dimensional
algebraic subshifts over $\FF$ defined by a polynomial without line polynomial factors in
more than one direction. We also find an algebraic subshift that is defined by a product of
two line polynomials that has this property (the 4-dot system) and another one that does not.
\end{abstract}

\section{Introduction and preliminaries}

The European Association for Theoretical Computer Science, EATCS, celebrated its 25th anniversary at ICALP 1997 in Bologna. The keynote address was given by Maurice Nivat. His talk --
titled  ``Towards a History of Automata: Why Were They Introduced? What Are They Good for?'' -- suggested the following intriguing problem that has become known as Nivat's conjecture: If an infinite two-dimensional grid has been colored in such a way that, for some $n,m\in\N$, the number of distinct $n\times m$ patterns is at most $nm$, is the coloring necessarily periodic in some direction? The conjecture has attracted wide interest, but it has turned out to be difficult and remains still unsolved today.

A number of partial results and related observations have been made. Periodicity has been established if the number of $n\times m$ patterns is at most $\alpha nm$ for progressively larger and larger constants $\alpha<1$: First for $\alpha=1/144$ in~\cite{epifanio}, then for $\alpha=1/16$ in~\cite{zamboni}, and finally for the best known constant $\alpha=1/2$ in~\cite{cyrkra}.
It is also known that having at most $2n$ patterns of size $2\times n$ and having at most $3n$ patterns of size $3\times n$ imply periodicity~\cite{sander,cyrkra16}.
In~\cite{kariszabados} we introduced an algebraic approach that leads to the following asymptotic result: A non-periodic coloring of the grid can only have finitely many pairs $(n,m)$ such that the number of distinct $n\times m$ patterns is at most $nm$. Of course Nivat's conjecture claims
that there are no such pairs.

In this paper we continue with our algebraic approach and prove that Nivat's conjecture holds on colorings
that come from certain algebraically defined sets, including the Ledrappier subshift~\cite{ledrappier}. In fact we prove a stronger statement that
in these sets periodicity is implied if for any shape -- not necessarily a rectangle -- the number of patterns of that shape is at most
the size of the shape. Even more: if a coloring has any non-trivial annihilator over $\Z$ (terms explained below) then it is periodic.

To make the statement precise we
recall some terminology. For a finite alphabet $A$, colorings $c\in A^{\Z^2}$ of the two-dimensional grid by elements of $A$ are called (two-dimensional) \emph{configurations}. Typically we use notation $c_{\vec{n}}$ for the color $c(\vec{n})\in A$ of cell $\vec{n}\in\Z^2$.
Basic operations on $A^{\Z^2}$ are \emph{translations\/}:
for any $\vec{t}\in\Z^2$ the translation $\tau_{\vec{t}}:A^{\Z^2}\longrightarrow A^{\Z^2}$ by $\vec{t}$ is defined by $\tau_{\vec{t}}(c)_{\vec{n}}=c_{\vec{n}-\vec{t}}$, for all $c\in A^{\Z^2}$ and all $\vec{n}\in\Z^2$.
We call a configuration $c$ \emph{periodic\/} if $\tau_{\vec{t}}(c)=c$ for some non-zero $\vec{t}\in\Z^2$, and we call
$\vec{t}$ a \emph{vector of periodicity}. If there are two linearly independent vectors of periodicity then $c$ is \emph{two-periodic}.
In this case it is easy to see that there are horizontal and vertical vectors of periodicity $(k,0)$ and $(0,k)$ for some $k\neq 0$, and consequently a vector of periodicity in every rational direction. We call $c$ \emph{one-periodic} if it is periodic but not two-periodic.

Colorings $p\in A^D$ of a finite shape $D\subset \Z^d$ are called \emph{$D$-patterns}, or simply patterns. The set of $D$-patterns that appear in a configuration $c$ is denoted by $P(c,D)$, that is,
$$P(c,D)=\{\tau_{\vec{t}}(c)_{|D}\ |\ \vec{t}\in\Z^2\ \}.$$
We say that $c$ has \emph{low complexity\/} with respect to shape $D$ if $|P(c,D)|\leq |D|$,
and we call $c$ a \emph{low complexity configuration\/} if it has low complexity with respect to some finite $D$.
Otherwise $c$ is a \emph{high complexity configuration}.
Low complexity configurations come up naturally in practical setups, e.g. in crystallography
particles may attach to each other only in very limited ways. Moreover, there are interesting mathematical problems associated to them, including
the famous Nivat's conjecture discussed above:
\begin{conjecture}[Maurice Nivat 1997]
Let $c\in A^{\Z^2}$ be a two-dimensional configuration. If $c$ has low complexity with respect to some rectangle $D=\{1,\dots,n\}\times\{1,\dots,m\}$ then $c$ is periodic.
\end{conjecture}
\noindent
The analogous claim  in dimensions higher than two fails, as does an analogous claim in two dimensions
for many other shapes than rectangles~\cite{cassaigne}.

To study the conjecture algebraically we replace the colors by integers, or elements of some other integral domain $R$, and express the
configuration $c$ as a formal power series $c(X,Y)$ over two variables $X$ and $Y$ in which the coefficient of monomial $X^iY^j$ is $c_{i,j}$, for all $i,j\in\Z$. Note that the exponents of the variables range from $-\infty$ to $+\infty$. In the following also polynomials may have negative powers of variables so all polynomials considered are actually Laurent polynomials. Let us denote by $R[X^{\pm 1}, Y^{\pm 1}]$ and $R[[X^{\pm 1}, Y^{\pm 1}]]$ the sets of such polynomials and power series, respectively, with coefficients in domain $R$.
We call a power series $c\in R[[X^{\pm 1}, Y^{\pm 1}]]$ \emph{finitary} if its coefficients take only finitely many different values. Since we color the grid using finitely many colors,
configurations are identified with finitary power series.

Multiplying configuration $c\in R[[X^{\pm 1}, Y^{\pm 1}]]$
by a monomial corresponds to translating it, and the periodicity of the configuration by vector $\vec{t}=(n,m)$ is then
equivalent to $(X^nY^m-1)c=0$, the zero power series. More generally, we say that polynomial $f\in R[X^{\pm 1}, Y^{\pm 1}]$ \emph{annihilates} power series $c$ if the formal product $fc$ is the zero power series.
The set of polynomials that annihilates a power series is a polynomial ideal, and is denoted by
\[ \Ann_R(c) = \{ f\in R[X^{\pm 1}, Y^{\pm 1}] ~|~ fc=0 \} .\]

We observed in~\cite{kariszabados} that, in the case $R=\Z$, if a configuration
has low complexity with respect to some shape then it is annihilated by some non-zero polynomial $f\neq 0$. The proof works unchanged
for an arbitrary field $R$.

\begin{lemma}[\cite{kariszabados}]
\label{th:low_complexity}
Let $R$ be a field or $R=\Z$. Let $c\in R[[X^{\pm 1}, Y^{\pm 1}]]$ be a low complexity configuration.
Then $\Ann_R(c)$ contains a non-zero polynomial.
\end{lemma}

One of the main results of~\cite{kariszabados} states that, in the case $R=\Z$, if a
configuration $c$ is annihilated by a non-zero polynomial then it has annihilators of particularly nice form:
\begin{theorem}[\cite{kariszabados}]
  \label{th:decompo}
  Let $c\in\Z[[X^{\pm 1}, Y^{\pm 1}]]$ be a configuration (a finitary power series) annihilated by some non-zero polynomial.\\
  Then there exists non-zero $(i_1, j_1), \hdots, (i_m, j_m)\in\Z^2$ such that
  \[ (X^{i_1}Y^{j_1} - 1) \cdots (X^{i_m}Y^{j_m} - 1) \in \Ann_{\Z}(c) .\]
\end{theorem}
\noindent
Unlike Lemma~\ref{th:low_complexity} this theorem does not generalize to all cases when $R$ is a field.

For a polynomial $f = \sum a_{i,j} X^iY^j$, we call $\text{supp}(f) = \{ (i,j) ~|~ a_{i,j}\neq 0 \}$ its
\emph{support}. A \emph{line polynomial} is a polynomial with its terms aligned all on the same line: $f$ is a line polynomial in direction $\vec{u}\in\Z^2\setminus\{\vec{0}\}$ if and only if $\text{supp}(f)$ contains at least two elements and for some $\vec{n}\in\Z^2$ we have $\text{supp}(f)\subseteq \{\vec{n}+r\vec{u}\ |\ r\in\Q\}$.
Note that the annihilator provided by Theorem~\ref{th:decompo} is a product of line polynomials. A central feature of line polynomials is that any
configuration that is annihilated by a line polynomial is periodic in the direction of the line polynomial~\cite{kariszabados}.

If $R$ is a finite field and if $f\in R[X^{\pm 1}, Y^{\pm 1}]$ is a non-zero polynomial then we define the set
$$
X_f=\{ c\in R[[X^{\pm 1}, Y^{\pm 1}]]\ |\ fc=0\}
$$
of
all configurations that $f$ annihilates
and call it the \emph{algebraic subshift} defined by $f$. It is a subshift of finite type (SFT, see~\cite{lindmarcus} for the symbolic dynamics terminology). We have that $c\in X_f\Longleftrightarrow f\in\Ann_R(c)\Longleftrightarrow fc=0$.

By Lemma~\ref{th:low_complexity} all low complexity configurations belong to algebraic subshifts when the symbols are renamed as elements of
a finite field. It is thus enough to consider Nivat's conjecture on elements of algebraic subshifts: if the conjecture is false then there is a counter example configuration
that belongs to an algebraic subshift. This is the approach taken in this work. We consider various algebraic subshifts and see whether low complexity implies periodicity
in these subshifts. When this is the case we say that the subshift has the \emph{generalized Nivat property}.
We prove that an algebraic subshift $X_f$ has the generalized Nivat property if the defining polynomial $f$ does not have line polynomial factors in two different directions.

The paper is organized as follows. In Section~\ref{sec:ledrappier} we consider as a particular example the well known Ledrappier subshift,
also known as the 3-dot system, and prove that any low complexity configuration in the Ledrappier subshift is periodic. In Section~\ref{sec:general}
we generalize this result to all algebraic subshifts where the defining polynomial does not have line polynomial factors in more than one direction. This includes, for example, the
space-time diagrams of all additive cellular automata over finite fields. In Section~\ref{sec:fourdot} we consider the 4-dot system defined by polynomial $(1+X)(1+Y)$
over $\F$ and prove that it also has the generalized Nivat property although the defining polynomial has horizontal and vertical line polynomial factors,
while the analogous system defined by $(1+X^2)(1+Y^2)$ does not have the property.

\section{The Ledrappier Subshift}
\label{sec:ledrappier}

The Ledrappier subshift, or the 3-dot system,  consists of the configurations over the binary alphabet $\F$
that are annihilated by the polynomial $1+X+Y \in \F[X^{\pm 1}, Y^{\pm 1}]$.

\begin{definition}[\cite{ledrappier}]
  The \emph{Ledrappier subshift} $L$ is the algebraic subshift $X_{f_L}$ over $\F$ defined by the annihilator $f_L = 1+X+Y$.
\end{definition}

In this work we relate the generalized Nivat property to the number of line polynomial factors of the defining annihilator.
It turns out that $f_L$ has none.

A visual way of seeing if a polynomial has a line polynomial factor is to look at the shape of the convex hull of its support: the convex hull has
parallel sides in the directions of its line polynomial factors.
To be more precise, we need some definitions of discrete geometry.
The closed half plane in a direction $\vec{v}\in\Z^2\setminus\{\vec{0}\}$ is the set
$\overline{H}_{\vec{v}}=\{\vec{x}\in\Z^2 |\ \vec{x}\cdot\vec{v}\geq 0\}$, and the open half plane
$H_{\vec{v}}$ is defined analogously to contain those $\vec{x}\in\Z^2$ that satisfy $\vec{x}\cdot\vec{v} > 0$. The boundary of the half plane
is $\overline{H}_{\vec{v}}\setminus H_{\vec{v}}$.
We say that a finite set $D\subseteq\Z^2$ has an \emph{outer edge} perpendicular to $\vec{v}\in\Z^2\setminus\{\vec{0}\}$ if there is $\vec{x}\in D$ such that $D\subseteq \vec{x}+\overline{H}_{\vec{v}}$ and there are at least two elements
of $D$ on the boundary $\vec{x}+(\overline{H}_{\vec{v}}\setminus H_{\vec{v}})$. See \cref{fig:halfplane} for an illustration.

\begin{figure}[ht]
\begin{center}
\includegraphics[scale=0.4]{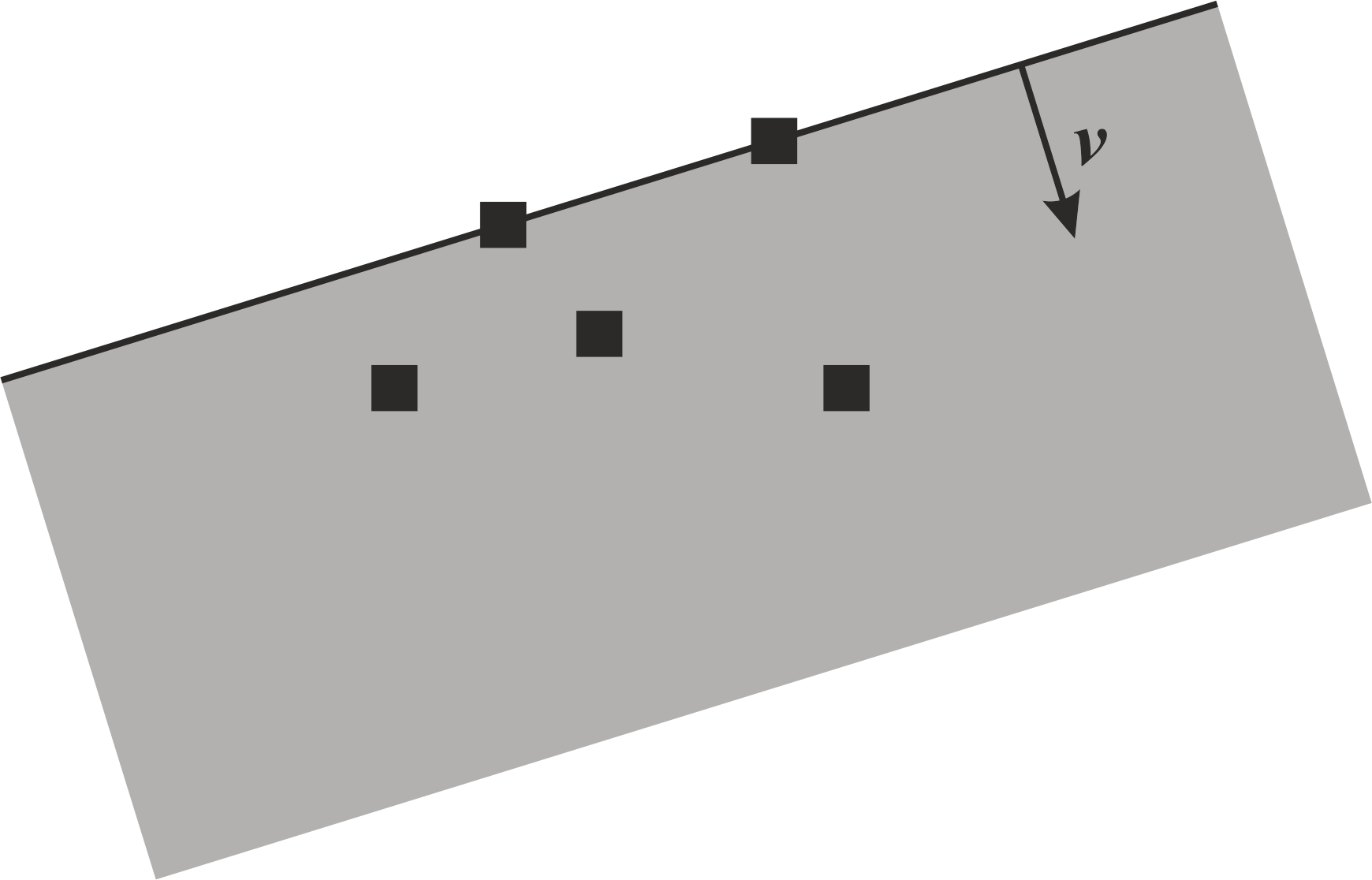}
\end{center}
  \caption{The set of black cells has an outer edge perpendicular to vector $\vec{v}$.}
  \label{fig:halfplane}
\end{figure}

\begin{lemma}
\label{lem:halfplane}
Let $g,h$ be non-zero polynomials such that $\text{supp}(g)$ has an outer edge perpendicular to $\vec{v}$. Then also
$\text{supp}(gh)$ has an outer edge perpendicular to $\vec{v}$.
\end{lemma}

\begin{proof}
Let vector $\vec{u}\in\Z^2\setminus\{\vec{0}\}$ be perpendicular to $\vec{v}$. By the hypotheses of the lemma
there are $\vec{x}\in \Z^2$ and a line polynomial $\alpha$ in the direction $\vec{u}$ such that
$\text{supp}(g)\subseteq \vec{x}+\overline{H}_{\vec{v}}$ and
$\text{supp}(g-\alpha)\subseteq \vec{x}+H_{\vec{v}}$. Here $\alpha$ contains the terms of $g$ along the boundary of the half plane $\vec{x}+H_{\vec{v}}$.
Analogously, for any non-zero polynomial $h$ there exists $\vec{y}\in \Z^2$ and polynomial $\beta\neq 0$ that is either
a monomial or a line polynomial in the direction $\vec{u}$ such that
$\text{supp}(h)\subseteq \vec{y}+\overline{H}_{\vec{v}}$ and
$\text{supp}(h-\beta)\subseteq \vec{y}+H_{\vec{v}}$. But then
$\text{supp}(gh)\subseteq \vec{x}+\vec{y}+\overline{H}_{\vec{v}}$
and $\text{supp}(gh-\alpha\beta)\subseteq \vec{x}+\vec{y}+H_{\vec{v}}$.
Because $\alpha\beta$ is a line polynomial in the direction $\vec{u}$,
this proves that the support of
$gh$ has an outer edge perpendicular to $\vec{v}$.
\end{proof}

\begin{corollary}
If $f\neq 0$ has a line polynomial factor in the direction $\vec{u}$
then $\text{supp}(f)$ has outer edges perpendicular to $\vec{v}$ and $-\vec{v}$, where $\vec{v}$ is a
vector perpendicular to $\vec{u}$.
\end{corollary}

\begin{proof}
A line polynomial $g$ in the direction $\vec{u}$ has outer edges perpendicular to $\vec{v}$ and $-\vec{v}$. The claim then follows directly from Lemma~\ref{lem:halfplane}.
\end{proof}

Because the convex hull of $\text{supp}(f_L)$ is a triangle, it does not have parallel outer edges and therefore
no line polynomial factors.

\begin{corollary}
  \label{th:ledrappier_no_line_poly}
  Polynomial $f_L = 1+X+Y$ has no line polynomial factors.
\end{corollary}

\noindent
{\bf Remark.} In fact, it is very easy to see that $f_L$ is an irreducible polynomial, which directly implies Corollary~\ref{th:ledrappier_no_line_poly}.
\medskip

\noindent
Now we can prove that the Ledrappier subshift has the generalized Nivat property.

\begin{theorem}
  \label{th:ledrappier}
  Any low complexity $c\in X_{f_L}$ is periodic.
\end{theorem}

\begin{proof}
In this proof we are going to need annihilators of $c$ over $\F$ and over $\Z$. We interpret $c$ as a configuration over $\Z$ using the renaming $\F\longrightarrow \Z$ that maps $0_\F\mapsto 0_\Z$ and  $1_\F\mapsto 1_\Z$. We write
$\Ann_\Z(c)$ for the annihilator ideal of the configuration over $\Z$ obtained after this renaming.
We are going to prove the following stronger statement: if $\Ann_\Z(c)$ contains a non-zero polynomial then $c$ is  periodic. The result then follows from Lemma~\ref{th:low_complexity}.

Let $c\in L$ be non-periodic. We first prove that $\Ann_\F(c)$ is the principal ideal generated by $f_L$, that is,
for every $g\in \Ann_\F(c)$ there exists $\alpha\in\F[X^{\pm 1}, Y^{\pm 1}]$ such that $g=\alpha f_L$.
First, let $g\in \Ann_\F(c)$ be a proper polynomial which means that all exponents of variables are non-negative.
Because $X = 1+Y+f_L$ we can eliminate variable $X$: there exists a polynomial $\alpha(X,Y)$ and a polynomial $\beta(Y)$ in variable $Y$ only such that $g = \alpha(X,Y)f_L+\beta(Y)$.
Because both $f_L$ and $g$ are annihilators of $c$, $\beta(Y)\in\Ann_\F(c)$ as well.

If $\beta\neq 0$ then it is either a single monomial (in which case $c=0$) or a line polynomial annihilator of $c$. Any configuration with a line polynomial annihilator is periodic in the direction of the polynomial. We conclude that $\beta=0$ so that $g = \alpha f_L$, as claimed.

Consider then arbitrary $g\in \Ann_\F(c)$ with possibly some negative exponents.
Because $g=X^iY^jg'$ for some $i,j\in\Z$ and a proper polynomial $g'\in\Ann_\F(c)$, we conclude that also in this case  $g$ is a multiple of $f_L$, and therefore we have that
 $\Ann_\F(c)$ is the principal ideal generated by $f_L$.

Now it remains to show that $\Ann_\Z(c)$, the set of annihilators over $\Z$, is trivial. Suppose by contrary that
there is a non-zero annihilator in $\Ann_\Z(c)$. By Theorem~\ref{th:decompo}
there exists non-zero $(i_1, j_1), \hdots, (i_m, j_m)\in\Z^2$ such that
$(X^{i_1}Y^{j_1} - 1) \cdots (X^{i_m}Y^{j_m} - 1) \in \Ann_{\Z}(c)$.
By performing computations modulo two instead, we have that
$(X^{i_1}Y^{j_1} - 1) \cdots (X^{i_m}Y^{j_m} - 1) \in \Ann_{\F}(c)$.
But then this polynomial, which only has line polynomial factors, is a multiple of $f_L$.
All factors of line polynomials are line polynomials in the same direction (or monomials) so that all
irreducible factors of $f_L$ are line polynomials, contradicting Corollary~\ref{th:ledrappier_no_line_poly}.
Here we used the fact that every polynomial can be factored uniquely into its irreducible factors~\cite{cox}.

\end{proof}

\section{Annihilators with Few Line Polynomial Factors}
\label{sec:general}
One of the key element in the proof of \cref{th:ledrappier} is the fact that we can find an annihilator with one variable eliminated from it.
It turns out that this is something we can do with other annihilators than $f_L$, and to do so we will use some parts of \textit{elimination theory}, so-called
resultants. Everything needed to understand this paper will be introduced but for more, Cox's book \cite{cox} is a good introduction.

\subsection{Resultant}
Usually resultant are defined for proper polynomials, so in order to stick with usual definitions we define them likewise. We will still use them for talking about Laurent polynomials, which is not a problem since from a Laurent polynomial annihilator we can always obtain a proper one by multiplying it by a suitable monomial.

Let $\K$ be a field, e.g. $\K=\FF$ for a prime $p$, and
let $K = \K[Y]$ be the ring of polynomials with one variable $Y$. Then $K[X] =\K[X,Y]$ is the ring of polynomials with two variables over $\K$.
Resultant can be defined for polynomials with more than two variables, but two is all we need here.
Resultant can be defined formally as a determinant of some matrix with coefficients in $K$, but for our purpose, all we need to know is that for any $f,g\in \K[X,Y]$, there exists a polynomial $\Res(f,g)\in \K[Y]$ in variable $Y$ only with the following two properties:

\begin{prop}
  \label{prop:det0}
  The polynomials $f$ and $g$ have a common factor in $\K[X,Y]$ if and only if $\Res(f,g) = 0$.
\end{prop}

\begin{prop}
  \label{prop:det}
  There exist $\alpha, \beta\in \K[X,Y]$ such that
  \[ \alpha f + \beta g = \Res(f,g) .\]
\end{prop}

The second property exactly states that we can find a linear combination of $f$ and $g$ that eliminates variable $X$.

\begin{lemma}
  \label{th:two_ann}
  Let $c$ be a power series over a field $R$. If $c$ is annihilated by two non-zero polynomials $f$ and $g$ with no common factors then $c$ is two-periodic.
\end{lemma}

\begin{proof}
Using \cref{prop:det}, let $\alpha, \beta\in \K[X,Y]$ be such that
\[ \alpha f + \beta g = \Res(f,g) .\]
Because $f$ and $g$ are both annihilators of $c$ so is $\Res(f,g)$.
Because $f$ and $g$ have no common factors, \cref{prop:det0} tells us that $\Res(f,g)\neq 0$.
We found an annihilator which is a line polynomial or a monomial (only with variable $Y$) so $c$ must be periodic along $Y$.

Symmetrically, there also exist $\gamma, \delta\in \K[X,Y]$ such that
\[ \gamma f + \delta g = \ResY(f,g) .\]
Again,  $\ResY(f,g)\neq 0$, so $c$ is also periodic along $X$.
\end{proof}

\noindent
{\bf Remark.}
In~\cite{karifull} we analyzed the structure of annihilator ideals of configurations over $\Z$, and an analogous result for the case of $\Z$  follow from those considerations.
We included here the simple proof above for the sake of completeness, and in order to cover also the case of arbitrary coefficient fields.

\subsection{Main Result}

Now we can easily generalize \cref{th:ledrappier}. We consider an algebraic subshift $X_f$ over a finite field $\FF$ defined by an annihilator
$f\in \FF[X^{\pm 1}, Y^{\pm 1}]$. A configuration $c\in \FF[[X^{\pm 1}, Y^{\pm 1}]]$ will also be interpreted as a configuration over $\Z$
by mapping the symbols by $a_{\FF} \mapsto a_{\Z}$ for all $a\in\{0,1,\dots ,p-1\}$. Then, as in the proof of \cref{th:ledrappier}, we can define both
$\Ann_\FF(c)$ and $\Ann_\Z(c)$, and use the fact that any $g\in\Ann_\Z(c)$ is also in $\Ann_\FF(c)$ when its coefficients are reduced modulo $p$.

\begin{theorem}
  \label{th:main_th}
  Let $c\in X_f$ for a polynomial $f\in \FF[X^{\pm 1}, Y^{\pm 1}]$, and suppose that $\Ann_\Z(c)$ contains a non-zero polynomial.
  \begin{itemize}
    \item If $f$ has no line polynomial factors then $c$ is two-periodic.
    \item If all line polynomial factors of $f$ are in the same direction then $c$ is periodic in this direction.
  \end{itemize}
\end{theorem}
\begin{proof}
By Theorem~\ref{th:decompo}
there exists non-zero $(i_1, j_1), \hdots, (i_m, j_m)\in\Z^2$ such that
$(X^{i_1}Y^{j_1} - 1) \cdots (X^{i_m}Y^{j_m} - 1) \in \Ann_{\Z}(c)$.
By performing computations modulo $p$ instead, we have that
$g(X,Y)=(X^{i_1}Y^{j_1} - 1) \cdots (X^{i_m}Y^{j_m} - 1) \in \Ann_{\FF}(c)$.

If $f$ has no line polynomial factors then $f$ and $g$ do not have any common factors. By Lemma~\ref{th:two_ann}
then $c$ is two-periodic. This proves the first claim.

Suppose then that all line polynomial factors of $f$ are in the same direction. Let $h$ be the greatest common divisor of $f$ and $g$
so that we can write $f=f'h$ and $g=g'h$ where $f'$ and  $g'$ do not have common factors. Note that $h$ is a line polynomial: as a factor of $g$
$h$ is a product of line polynomials, and as a factor of $f$ these line polynomials are all in the same direction.

Because $c'=hc$ is annihilated by both $f'$ and $g'$ it follows from Lemma~\ref{th:two_ann}
that $c'$ is two-periodic. In particular, there is a line polynomial $h'$ in the direction of $h$ that annihilates $c'$.
We have $hh'\in \Ann_{\FF}(c)$ so that $c$ is annihilated by the line polynomial $hh'$ and is therefore periodic in this direction.

\end{proof}

Using Lemma~\ref{th:low_complexity} we now immediately get that algebraic subshifts defined by an annihilator with at most one line polynomial factor
have the generalized Nivat property.

\begin{corollary}
  \label{coro:alg_sft}
  Let $c\in X_f$ for a polynomial $f\in \FF[X^{\pm 1}, Y^{\pm 1}]$ whose line polynomial factors are all in the same direction.
  If $c$ has low complexity then it is periodic.
\end{corollary}

\noindent
{\bf Remark.}
Elements of the Ledrappier subshift are exactly the space-time diagrams of the one-dimensional XOR cellular automaton. More generally,
the algebraic subshift $X_f$ defined by a polynomial of the form $f(X,Y)=Y-g(X)\in \FF[X^{\pm 1}, Y^{\pm 1}]$ consists exactly of the
space-time diagrams of the one-dimensional additive cellular automaton over $\FF$ whose local rule is given by $g(X)$. See Section~9 of~\cite{karisurvey} for a short discussion about additive cellular automata. If $g(X)$ has at least
two terms then the support of $f$ has triangular shape and therefore $f$ has no line polynomial factors. If $g(X)$ has one term then $f$ itself is
a line polynomial. In any case, Theorem~\ref{th:main_th} and Corollary~\ref{coro:alg_sft} hold for the space-time diagrams of all
one-dimensional additive cellular automata over field $\FF$.

\section{Square Annihilators}
\label{sec:fourdot}
After Corollary~\ref{coro:alg_sft}, it is natural to take a look at configurations annihilated by polynomials with line polynomial factors in more directions. It turns out that already products of two line polynomials include examples with and without the generalized Nivat property.
We first prove that the 4-dot system defined by $(1+X)(1+Y)$ over $\F$ has the generalized Nivat property and then we show that
the system defined by $(1+X^2)(1+Y^2)$ does not.

\subsection{The 4-dot system}

\begin{definition}
  The \emph{4-dot system} $S$ is the algebraic subshift $X_{f_S}$ over $\F$ defined by the annihilator $f_S =  1+X+Y+XY = (1+X)(1+Y)$.
\end{definition}

\begin{theorem}
  \label{th:square}
  Every low complexity $c\in X_{f_S}$ is periodic.
\end{theorem}

\begin{proof}
Similarly as before, we are going to prove the more general statement that if $c$ has a non-trivial annihilator $p_0$ over $\Z$ then it is periodic.

We first observe that $c=h+v$ for $h,v\in \F[[X^{\pm 1}, Y^{\pm 1}]]$  that are
$(1,0)$-periodic and $(0,1)$-periodic, respectively. Indeed, we can take $h_{i,j}=c_{0,j}$ and $v_{i,j}=c_{i,0}+c_{0,0}$, for all $i,j\in \Z$.
Because $(1+X^i)(1+Y^j)$ is a multiple of $(1+X)(1+Y)$ over $\F$, polynomial $(1+X^i)(1+Y^j)$ annihilates $c$, for all $i,j\in\Z$. This
means that $c_{i,j}=c_{0,j}+c_{i,0}+c_{0,0}=h_{i,j}+v_{i,j}$.

Using the periodicity of $h$ and $v$  we can write $h = \ones(X) s(Y)$ and $v = t(X) \ones(Y)$, with $\ones(X) = \sum_{i\in\Z} X^i$ and $s,t$ two formal series depending only one one variable.
Let us define another binary configuration $d$ by
\[ d(X,Y) = t(X)s(Y) .\]
In other words, $d$ is the configuration that has ones where both $h$ and $v$ have ones:
\[ d_{i,j} =
\begin{cases}
  1 \text{ if } h_{i,j} = v_{i,j} = 1\\
  0 \text{ otherwise}
\end{cases}
.\]
Interpreted in \Z, we have
\[ c = h+v - 2d .\]
This is the case since the two sides are identical modulo two and both sides only contain values 0 and 1.

Consider next the polynomial
\[ p = p_0 (X-1)(Y-1)\]
over $\F$. Because $p_0$ annihilates $c$, and $X-1$ and $Y-1$ annihilate $h$ and  $v$, respectively, we have that $pc=ph=pv=0$. Therefore
$pd=0$ as well, which can be written as
\[ p(X,Y) t(X) s(Y) = 0, \]
emphasizing the variable dependencies of the polynomials. We have the following two cases:
\medskip

\noindent
\emph{Case 1:} Suppose that $p(X,Y) t(X)=0$. In $p(X,Y)$ we collect together terms with the same power of variable $Y$, obtaining
$$
p(X,Y)=\sum_{j\in \Z} Y^jp_j(X)
$$
where at least some $p_j(X)$ is a non-zero polynomial. We have
$$\sum_{j\in \Z} Y^jp_j(X)t(X)=0.$$
This is an identity of formal power series so that
$p_j(X)t(X)=0$ for all $j\in\Z$. But then also $p_j(X)v=p_j(X)t(X)u(Y)=0$, so that $v$ is annihilated by a non-zero
horizontal line polynomial (or a non-zero monomial) $p_j(X)$. We conclude that
$v$ is horizontally periodic. But then also $c=h+v$ is horizontally periodic as a sum of two horizontally periodic configurations.

\medskip

\noindent
\emph{Case 2:} Suppose that $p(X,Y) t(X)\neq 0$. Now we collect in $p(X,Y) t(X)$ together variables with the same power of variable $X$, obtaining
$$
p(X,Y) t(X) = \sum_{i\in \Z} X^i q_i(Y),
$$
where at least some $q_i(Y)$ is a non-zero polynomial. Note that all $q_i(Y)$ are polynomials because powers of the variable
$Y$ only come from the polynomial $p(X,Y)$. Because
$$
0=p(X,Y)d(X,Y)=p(X,Y) t(X)s(Y) = \sum_{i\in \Z} X^i q_i(Y)s(Y),
$$
we have that $q_i(Y)s(Y)=0$ for all $i\in\Z$. Analogously to case 1 above, this implies that $h$ is vertically periodic, and therefore also $c$ is vertically periodic.
\end{proof}

\subsection{An algebraic subshift without the generalized Nivat property}

For some polynomials with two line polynomial factors, the associated subshift does not have the generalized Nivat property. This is typically the case when the annihilating polynomial allows the use of sublattices.

\begin{theorem}
  \label{th:square2}
  There exists a configuration $c$ over $\F$ annihilated by $f_T = (1+X^2)(1+Y^2)$ which is not periodic but has low complexity.
\end{theorem}

\begin{proof}
Let us take $c = h + v$, with
\begin{center}
  $ h_{i,j} =
  \begin{cases}
    1 \text{ if } j = 0 \text{ and } i \text{ even} \\
    0 \text{ otherwise}
  \end{cases}
  $
  and ~
  $ v_{i,j} =
  \begin{cases}
    1 \text{ if } i = 1 \text{ and } j \text{ even} \\
    0 \text{ otherwise}
  \end{cases}
  .$
\end{center}
Visually, $c$ is the superposition of a horizontal and a vertical line on two disjoint sublattices, see \cref{fig:sublat}.
Clearly $h$ is one-periodic with periodicity vector $(2,0)$ and $v$ is one-periodic with periodicity vector $(0,2)$. Their sum  $c$ is not periodic.

\begin{figure}[ht]
  \centering
  \scalebox{0.6}{\begin{tikzpicture}

  \definecolor{redd}{RGB}{255,59,37}
  \definecolor{bluee}{RGB}{17,86,178}

  \newcommand*{\rectopacity}{1}
  \newcommand*{\nbopacity}{0.8}

\draw [step=1, opacity=0.5] (0,0) grid (11,11);

\foreach \i in {0, 2, ..., 10}{
	\draw[fill=redd, opacity=\rectopacity, color=redd] (\i, 4) rectangle (\i+1, 4+1) ;
}

\foreach \j in {0, 2, ..., 10}{
	\draw[fill=bluee, opacity=\rectopacity, color=bluee] (3, \j) rectangle (3+1, \j+1) ;
}

\foreach \i in {0, 2, ..., 10}{
\foreach \j in {0, 2, ..., 10}{
	\node [opacity=\nbopacity] at (\i+0.5, \j+0.5) {1};
}
}

\foreach \i in {1, 3, ..., 10}{
\foreach \j in {0, 2, ..., 10}{
	\node [opacity=\nbopacity] at (\i+0.5, \j+0.5) {2};
}
}

\foreach \i in {0, 2, ..., 10}{
\foreach \j in {1, 3, ..., 10}{
	\node [opacity=\nbopacity] at (\i+0.5, \j+0.5) {3};
}
}

\foreach \i in {1, 3, ..., 10}{
\foreach \j in {1, 3, ..., 10}{
	\node [opacity=\nbopacity] at (\i+0.5, \j+0.5) {4};
}
}

\foreach \i in {3, 5, 7}{
  \foreach \j in {2,4,6}{
	 \draw[color=black, line width=3pt] (\i, \j) rectangle (\i+1, \j+1) ;
  }
}

\end{tikzpicture}}
  \caption{Sublattices of $c$ and shape $D$ superimposed. The horizontal line is from $h$ and the vertical one from $v$.}
  \label{fig:sublat}
\end{figure}

The periodicity of $h$ and $v$ directly implies that $f_T$ annihilates $c$:
$h$ being $(2,0)$-periodic $(1+X^2) h = 0$ and, analogously, $v$ being $(0,2)$-periodic $(1+Y^2) v = 0$. This means that
$f_Tc=(1+X^2)(1+Y^2)h+(1+X^2)(1+Y^2)v=0$.

\begin{figure}[ht]
  \centering
  \scalebox{0.6}{\begin{tikzpicture}

  \newcommand*{\opacity}{1}

\draw [step=1, opacity=0.3] (0,0) grid (4,4);
\draw  [opacity=0.3] (5,0) -- (5,5) -- (0,5);

\foreach \i in {0, 2, 4}{
  \foreach \j in {0, 2, 4}{
	 \draw[fill=black, opacity=\opacity, color=black] (\i, \j) rectangle (\i+1, \j+1) ;
  }
}

\end{tikzpicture}}
  \caption{The shape $D$ of low complexity.}
  \label{fig:D}
\end{figure}

The last thing we have to check is that $c$ has low complexity, i.e, there is a shape $D$ such that $P_c(D) \leq |D|$.
It is sufficient to take $D$ to be the scattered 3x3 square, as shown in \cref{fig:D}.
Patterns of shape $D$ in $c$ will only contain values from one of the four sublattices, depending on the parity of its position.
If $D$ is superimposed with sublattices 3 or 4, the pattern is blank.
With sublattice 1, it can only contain values from $h$, so it can have four different values: blank, and the horizontal line crossing at the top, the middle or the bottom. If it is on sublattice 2, then it has values from $v$, and here again there are four different possibilities.
Counting the blank shape only once, we obtain $P_c(D) = 1 + 3 + 3 = 7 < 9 = |D|$.

\end{proof}

\section{Conclusions}

To prove Nivat's conjecture it is enough to consider configurations of two-dimensional algebraic subshifts over $\FF$. We prove that simplest
subshifts do not contain a counter example to the conjecture, and they even have the stronger property that any configuration that has low complexity with respect to any finite shape is periodic. It remains open to characterize which algebraic subshifts have this generalized Nivat property. It remains also an interesting algorithmic problem to decide for a set of at most $|D|$ allowed patterns of shape $D$,
for some finite  $D\subseteq\Z^2$, whether there is a configuration all of whose $D$-patterns are among the allowed patterns. We do not know if
this emptyness problem of low complexity subshifts of finite type is decidable in general, but if the given patterns come from one of the algebraic subshifts studied
in this work then the emptyness problem is decidable by a standard argument because any configuration with these patterns must be periodic. We also ask
whether the generalized Nivat property may fail for other reasons than having the support of the defining polynomial in a proper sublattice of $\Z^2$.

\bibliographystyle{abbrv}
\bibliography{biblio}

\end{document}